\documentclass[3p, a4paper, 11pt]{elsarticle}

\usepackage{amssymb}
\usepackage{amsmath}
\usepackage{amsthm}
\usepackage{algorithm}
\usepackage{algorithmic}
\usepackage{float}
\usepackage[capitalize]{cleveref}
\usepackage{graphicx}
\usepackage{subcaption}

\journal{X}

\begin{document}
\begin{frontmatter}

\title{Two-sided Riemannian optimization model order reduction for linear systems with quadratic outputs} 

\author{Xiaolong Wang} 
\author{Chenglong Liu}
\author{Tongtu Tian}
\affiliation{organization={School of Mathematics and Statistics},
            addressline={Northwestern Polytechnical University}, 
            city={Xi'an},
            postcode={710129}, 
            state={Shaanxi},
            country={China}}
\begin{abstract}
This paper investigates structure-preserving $H_2$-optimal model order reduction 
(MOR) for linear systems with quadratic outputs.
Within a Petrov-Galerkin projection framework, the $H_2$-optimal MOR problem is first formulated as an optimization problem on the Grassmann manifold, for which a corresponding bivariable alternating optimization algorithm is proposed.
Furthermore, to explicitly guarantee the asymptotic stability of the reduced-order 
model, a second approach is introduced by imposing specific constraints on the 
projection matrices. We reformulate the problem as a novel optimization task on the 
Stiefel manifold and construct a corresponding solution algorithm.
The computational bottleneck in both iterative methods is addressed by developing 
an approximate solver for Sylvester equations based on orthogonal polynomial 
expansions, which significantly enhances the overall efficiency.
Numerical experiments validate that the obtained reduced models provide significant 
advantages in approximation accuracy and computational efficiency. 
\end{abstract}
\end{frontmatter}




\section{Introduction}
\label{section1}

Mathematical modeling of large scale complex systems is indispensable in contemporary fields like aerospace engineering, automatic control, computational fluid dynamics, and signal processing.
However, the high dimensionality of these models often presents challenges such as prohibitive computational complexity, extensive memory demands, and difficulties in satisfying real-time performance requirements, significantly impeding their engineering applications.
Model Order Reduction (MOR) techniques offer a potent solution to these issues.
MOR aims to construct low-dimensional approximate models that drastically reduce computational and storage burdens while preserving the essential dynamics and input-output characteristics of the original high-fidelity systems, thereby improving the tractability and efficiency of complex system analysis and design.
Over time, MOR has matured into a robust theoretical and applied discipline, with 
well  established methods such as balanced truncation, Krylov subspace 
techniques, and $H_2$-optimal approaches \cite{Benner2017, Antoulas2020}.
\par

The inherent structure of many dynamical systems carries significant physical meaning.
Consequently, structure-preserving MOR for systems with diverse characteristics, 
including negative imaginary, port-Hamiltonian and coupled systems, has become 
an active research focus.
Specifically, when observed quantities represent variances or state deviations relative to a reference, the system's output equation naturally assumes a quadratic structure, defining what are known as systems with quadratic outputs.
An established body of literature addresses MOR for linear dynamical systems with quadratic outputs (LQO systems).
A common intuitive strategy involves linearizing the system's output equation, thereby recasting the LQO system as an equivalent linear multi-output system.
This allows the application of standard MOR techniques for multi-output linear systems, such as Balanced Truncation \cite{VanBeeumen2010} and Krylov subspace methods \cite{Benner2015}.
However, this linearization often yields a linear system with a high-dimensional output space, potentially incurring significant computational expense.
An alternative innovative method, proposed in \cite{Pulch2019}, constructs a companion Quadratic Bilinear (QB) system.
The linear output of this QB system precisely replicates the quadratic output of the original LQO system.
Model order reduction is subsequently performed by applying balanced truncation to this QB system.
This process, however, ultimately yields a reduced-order QB system.
Direct reconstruction of the desired reduced-order LQO system from this resultant low-order QB system is often infeasible.
Reference \cite{Benner2021} defines Gramians and the $H_2$-norm for LQO systems.
Based on these definitions, \cite{Benner2021} also proposes a structure-preserving balanced truncation method.
Reference \cite{Gosea2019} extends the Iterative Rational Krylov Algorithm (IRKA) \cite{Gugercin2008} to LQO systems, resulting in an iterative interpolation-based MOR technique.
Subsequently, \cite{Gosea2022}, building on \cite{Gosea2019}, adapts the Adaptive 
Antoulas-Anderson algorithm to develop a data-driven modeling framework for LQO 
systems.
\par

Over the past few decades, $H_2$ model reduction methods have been extensively applied to structure-preserving MOR problems.
Riemannian optimization techniques have emerged as a highly effective approach for addressing $H_2$ optimal MOR problems.
In the context of $H_2$ optimal MOR for general dynamical systems, \cite{Sato2016} proposed a convergent algorithm based on product manifold geometry, while \cite{Zeng2015} introduced a bilateral iterative algorithm utilizing Grassmann manifolds.
For bilinear systems, \cite{Xu2019} and \cite{Xu2015} established $H_2$ optimal MOR algorithms on the Stiefel manifold and Grassmann manifold, respectively.
Furthermore, Riemannian optimization algorithms have also been developed for dynamical systems with other specific structures, such as coupled systems \cite{Yang2017,Sato2017,Yu2020,Wang2018,Jiang2019,Xu2022,Li2022}.
Additionally, several publications specifically explore the theoretical aspects of 
constructing iterative algorithms on Riemannian manifolds and their corresponding 
convergence analysis \cite{Ring2012,Sato2015,Sato2016a}.
\par

This paper investigates the structure-preserving $H_2$ optimal MOR problem for LQO systems.
Two primary approaches for constructing the coefficient matrices of reduced-order LQO systems are considered.
The first method involves constructing a two-variable alternating optimization algorithm on the Grassmann manifold $\mathrm{Gr}(n,r)$.
The core idea of this algorithm is to decompose the optimization problem into two 
sub-problems within each iteration, progressively approaching a locally optimal 
solution by alternately optimizing one variable while keeping the other fixed.
The second method imposes constraints on the projection transformation matrices, which theoretically ensures the stability of the reduced-order system.
Leveraging these strategies, the $H_2$ optimal MOR problem is reformulated as two distinct Riemannian optimization problems: one on the Grassmann manifold and the other on the Stiefel manifold.
Building on this, a bivariable alternating optimization algorithm and a Riemannian conjugate gradient model reduction method are established to generate reduced-order systems for the original LQO system.
A key challenge in the Riemannian conjugate gradient algorithm is the computational cost associated with computing the cost function's Riemannian gradient, which necessitates solving several Sylvester equations during each iteration.
As proposed in our previous paper, an efficient algorithm founded on Gramian approximation significantly enhances computational efficiency by circumventing the computationally intensive direct solution of Sylvester equations.
\par

The remainder of this paper is organized as follows.
In \cref{section2}, the state-space representation of the LQO system is presented, and its associated Gramians and $H_2$-norm are reviewed.
The structure-preserving $H_2$-optimal MOR problem for LQO systems is established in \cref{section3}.
In this section, the concept of the Grassmann manifold is introduced, allowing the problem to be reformulated as an unconstrained optimization task.
Building on this formulation, the necessary components for a Riemannian optimization framework are derived, and a corresponding iterative algorithm is proposed.
A fast algorithm designed to enhance the computational efficiency of the proposed methods is detailed in \cref{section4}.
Numerical simulation is presented in \cref{section5} to validate the effectiveness 
of the developed algorithms.
Finally, concluding remarks are provided in \cref{section6}.

\section{Problem statement}
\label{section2}

This section first presents a precise formulation of LQO systems.
Subsequently, objectives pertinent to MOR are reviewed, encompassing the Gramians 
and the $H_2$ norm of such systems.
Further details on LQO systems can be found in \cite{Gosea2019,Benner2021}.
\par

In this paper, we consider a linear time invariant system with a quadratic output, represented by the tuple of constant matrices $\Sigma = (A, B, C, M)$.
The system is described by the state-space equations
\begin{equation}
	{\Sigma}:\left\{\begin{array}{l}\dot{x}(t)=Ax(t)+Bu(t),\\y(t)=Cx(t)+x(t)^{\top}Mx(t),\end{array}\right.
\label{2.1}
\end{equation}
where $x(t) \in \mathbb{R}^n$ is the state vector, $u(t) \in \mathbb{R}^m$ is the input vector, and $y(t) \in \mathbb{R}$ is the scalar output;
these are defined for $t \in \left[0, t_{\mathrm{end}}\right]$.
The dimension of the matrices are $A\in\mathbb{R}^{n\times n}$, 
$B\in\mathbb{R}^{n\times m}$, $C\in\mathbb{R}^{1\times n}$ and 
$M\in\mathbb{R}^{n\times n}$.
It is assumed that the number of inputs is smaller than the state dimension (i.e., 
$m \ll n$).
Furthermore, the matrix $A$ is assumed to be Hurwitz, that is, all eigenvalues of 
$A$ have negative real parts, which ensures the asymptotic stability of the LQO 
systems. Throughout this paper, we assume that $M \in \mathbb{S}_n$, where 
$\mathbb{S}_n$ 
denotes the set of real symmetric matrices of size $n \times n$.
This is justified because for any matrix $M \in \mathbb{R}^{n \times n}$ and $x \in 
\mathbb{R}^n$, the quadratic form $x^{\top}Mx$ is equivalent to 
$\frac{1}{2}(x^{\top}Mx + x^{\top}M^{\top}x)$.
Note that system \eqref{2.1} is a multi-input single-output (MISO) system.
While LQO systems can be of the multi-input multi-output, the scope of this paper is 
focused on the MISO case.
\par

The primary objective is to construct a reduced-order model (ROM), denoted as 
$\hat{\Sigma} = (\hat{A}, \hat{B}, \hat{C}, \hat{M})$, of order $r$ ($r \ll n$), 
that 
approximates the original system $\Sigma$.
This ROM preserves the structural characteristics of \eqref{2.1} and is described by 
the following state-space representations
\begin{equation}
	\hat{\Sigma}:\left\{\begin{array}{l}\dot{\hat{x}}(t)=\hat{A} \hat{x}(t)+\hat{B} u(t), \\\hat{y}(t)=\hat{C}\hat{x}(t)+\hat{x}(t)^{\top}\hat{M}\hat{x}(t),\end{array}\right.
    \label{2.2}
\end{equation}
where $\hat{x}(t) \in \mathbb{R}^r$ is the reduced state vector, $u(t) \in \mathbb{R}^m$ is the input vector, and $\hat{y}(t) \in \mathbb{R}$ is the output.
The ROM matrices are defined as $\hat{A} \in \mathbb{R}^{r \times r}$, $\hat{B} \in \mathbb{R}^{r \times m}$, $\hat{C} \in \mathbb{R}^{1 \times r}$, and $\hat{M} \in \mathbb{S}_r$.
The ROM is to be constructed such that its output $\hat{y}(t)$ closely approximates the output $y(t)$ of the original system $\Sigma$ for all admissible inputs $u(t)$.
We aim to ensure two crucial properties of the ROM: its stability and the quadratic output structure.
\par

The state dynamics of system \eqref{2.1}, given by the equation $\dot{x}(t) = Ax(t) + Bu(t)$, are linear and time-invariant (LTI).
Consequently, the controllability Gramian $P$ for this system, which is determined by the LTI component $(A,B)$, is defined by the standard LTI system theory
\begin{equation*}
    P = \int_0^\infty e^{A\tau}BB^\top e^{A^\top\tau}\mathrm{d}\tau.
\end{equation*}
Given the prior assumption that the matrix $A$ is Hurwitz, this Gramian $P$ is the unique, symmetric, positive semidefinite solution to the algebraic Lyapunov equation
\begin{equation}\label{2.4}
    AP+PA^\top+BB^\top=0.
\end{equation}
\par

Building upon the preceding discussion, we associate Gramians with the distinct terms of the output equation in system \eqref{2.1}.
For the linear term $Cx(t)$, the standard observability Gramian $Q_1$, pertinent to the LTI pair $(A,C)$, is defined as
\begin{equation*}
    Q_1 = \int_0^\infty e^{A^\top\sigma}C^\top C e^{A\sigma}\mathrm{d}\sigma.
\end{equation*}
For the quadratic term $x(t)^\top Mx(t)$, a specialized Gramian $Q_2$ is introduced
\begin{equation*}
    Q_2 = \int_0^\infty\int_0^\infty
    e^{A^\top\sigma_1}Me^{A\sigma_2}B\left(e^{A^\top\sigma_1}Me^{A\sigma_2}B\right)^\top
    \mathrm{d}\sigma_1\mathrm{d}\sigma_2.
\end{equation*}
The Gramian $Q_2$ defined is referred to as the quadratic-output (QO) observability 
Gramian.
A total observability Gramian for the LQO system is then given by $Q = Q_1+Q_2$.
This composite Gramian $Q$ is the unique symmetric positive semidefinite solution to the generalized algebraic Lyapunov equation
\begin{equation*}
    A^\top Q+QA+C^\top C+MPM=0,
\end{equation*}
where $P$ is the controllability Gramian defined in \eqref{2.4}.
\par

The $H_2$ norm of the LQO system \eqref{2.1} is defined via its Volterra kernels as
\begin{equation}\label{2.8}
    \|\Sigma\|_{H_2} = \left(\int_0^\infty\|h_1(\sigma)\|_2^2\mathrm{d}\sigma+\int_0^\infty\int_0^\infty\|h_2(\sigma_1,\sigma_2)\|_2^2\mathrm{d}\sigma_1\mathrm{d}\sigma_2\right)^{\frac12},\nonumber
\end{equation}
where $h_1(\sigma)=Ce^{A\sigma}B$ and $h_2(\sigma_{1},\sigma_{2})=\mathrm{vec}\left(B^{\top}e^{A^{\top}\sigma_{1}}Me^{A\sigma_{2}}B\right)^{\top}$ are the linear and quadratic kernels, respectively.
Furthermore, the $H_2$ norm of $\Sigma$ can also be expressed in terms of the generalized observability Gramian $Q$ as \cite{Reiter2025}
\begin{equation}\label{2.9}
    \|\Sigma\|_{H_{2}} = \sqrt{\mathrm{tr}\left(B^{\top}QB\right)} = \sqrt{\mathrm{tr}\left(CPC^{\top}\right)+\mathrm{tr}\left(PMPM\right)}.\nonumber
\end{equation}

Consider the original system $\Sigma$ and a reduced-order system $\hat{\Sigma}$ when subjected to the same input $u(t)$.
An important relationship links the $L_{\infty}$ norm of the output error $y(t)-\hat{y}(t)$ to the $H_2$ norm of the error system $\Sigma_e := \Sigma-\hat{\Sigma}$ and the $L_2$ norm of the input \cite{Reiter2025}
\begin{equation}\label{2.10}
    \|y-\hat{y}\|_{L_\infty} = \sup_{t\geq0}|y(t)-\hat{y}(t)| \leq \|\Sigma-\hat{\Sigma}\|_{H_2} (\|u\|_{L_2}+\|u\otimes u\|_{L_2}).\nonumber
\end{equation}
This error bound highlights that minimizing $\|\Sigma-\hat{\Sigma}\|_{H_2}$ is a primary objective in model reduction.
A sufficiently small $H_2$ norm of the error system provides a guarantee for a small $L_\infty$ output error, given an input with finite $L_2$ norm.

\section{$H_2$ optimal MOR for LQO systems on Riemannian manifolds}
\label{section3}

This section investigates MOR problem for LQO systems under the Petrov-Galerkin 
projection framework.
Two novel Riemannian optimization methods are designed to address this problem.
First, we introduce the ROM for LQO systems within the Petrov-Galerkin projection 
setting. Then the $H_2$ optimal MOR is transformed into a Riemannian optimization 
problem, and an iterative algorithm is proposed.
Furthermore, we refine this algorithm by imposing constraints on the projection 
transformation matrices to actively preserve the stability of the reduced-order 
system, leading to a new Riemannian optimization algorithm.
\par

\subsection{$H_2$ optimal MOR problem}\label{section3.1}
Building upon the framework presented in Section~\ref{section2}, this subsection formulates the MOR problem for LQO systems as an $H_2$-norm optimization problem.
\par
The objective is to minimize the $H_2$ norm of the error system, $\|\Sigma-\hat{\Sigma}\|_{H_2}$, with respect to the reduced-order model parameters $(\hat{A},\hat{B},\hat{C},\hat{M})$, subject to the constraint that the reduced system matrix $\hat{A}$ is Hurwitz.
To this end, we define the cost function $J$ as the squared $H_2$ norm of the error system
\begin{equation}
    J(\hat{A},\hat{B},\hat{C},\hat{M}) = \|\Sigma-\hat{\Sigma}\|_{H_2}^2.\nonumber
\end{equation}
The error system $\Sigma_e = \Sigma-\hat{\Sigma}$, representing the difference between the full-order model \eqref{2.1} and the reduced-order model \eqref{2.2}, admits a state-space realization with matrices $(A_e,B_e,C_e,M_e)$ given by
\begin{equation}\label{3.1.1}
    (A_e,B_e,C_e,M_e) =
    \left(
    \begin{pmatrix}
        A & 0 \\
        0 & \hat{A}
    \end{pmatrix},
    \begin{pmatrix}
        B \\ \hat{B}
    \end{pmatrix},
    \begin{pmatrix}
        C & -\hat{C}
    \end{pmatrix},
    \begin{pmatrix}
        M & 0 \\
        0 & -\hat{M}
    \end{pmatrix}
    \right).
\end{equation}
It is evident that the output $y_e(t)$ of the error system $\Sigma_e$ described by \eqref{3.1.1} satisfies $y_e(t) = y(t)-\hat{y}(t)$.
Since the matrix $M$ of the original system is assumed to be symmetric, the matrix $\hat{M}$ of the reduced-order model should also be symmetric, i.e., $\hat{M} \in \mathbb{S}_r$.
The controllability Gramian $P_e$ and the generalized observability Gramian $Q_e$ for the error system $\Sigma_e$ are the unique solutions to the Lyapunov equations
\begin{align}
    A_eP_e+P_eA_e^\top+B_eB_e^\top&=0,\label{3.1.2}\\
    A_e^\top Q_e+Q_eA_e+C_e^\top C_e+M_eP_eM_e&=0,\label{3.1.3}.
\end{align}
Note that $A_e$ is Hurwitz is ensured by $A$ and $\hat{A}$ being Hurwitz.
To further analyze the $H_2$ norm of the error system, the Gramians $P_e$ and $Q_e$ are partitioned conformably with the block structure of $A_e$
\begin{equation}\label{3.1.4}
    P_e=
    \begin{pmatrix}
        P&X\\
        X^\top&\hat{P}
    \end{pmatrix},\quad
    Q_e=
    \begin{pmatrix}
        Q&Y\\
        Y^\top&\hat{Q}
    \end{pmatrix}.
\end{equation}
Here, $P$ and $Q$ are the Gramians of the full-order system $\Sigma$, while $\hat{P}$ and $\hat{Q}$ are the corresponding Gramians for the reduced-order model $\hat{\Sigma}$.
The matrices $X$ and $Y$ represent cross-coupling terms.
Substituting \eqref{3.1.1} and \eqref{3.1.4} into the Lyapunov equations \eqref{3.1.2} and \eqref{3.1.3} yields the following set of coupled matrix equations
\begin{align}
    AX+X{\hat{A}^\top}+B{\hat{B}^\top}&=0,\label{3.1.5}\\
    \hat{A}\hat{P}+\hat{P}{\hat{A}^\top}+\hat{B}{\hat{B}^\top}&=0,\label{3.1.6}\\
    A^\top Y+Y\hat{A}-C^\top\hat{C}-MX\hat{M}&=0,\label{3.1.7}\\
    {\hat{A}^\top}\hat{Q}+\hat{Q}\hat{A}+\hat{C}^\top\hat{C}+\hat{M}\hat{P}\hat{M}&=0\label{3.1.8}.
\end{align}
Using the partitioned form of $Q_e$ and the structure of $B_e$, the cost function $J(\hat{A},\hat{B},\hat{C},\hat{M})$ can be expressed as
\begin{equation}\label{3.1.9}
    J(\hat{A},\hat{B},\hat{C},\hat{M})=\mathrm{tr}\left(B_e^\top Q_eB_e\right)=\mathrm{tr}\left(B^\top QB+2B^\top Y\hat{B}+\hat{B}^\top\hat{Q}\hat{B}\right),
\end{equation}
or equivalently, using the partitioned form of $P_e$ and the structure of $C_e$, the cost function can also be expressed as
\begin{align}
    J(\hat{A}, \hat{B}, \hat{C}, \hat{M}) ={}& \mathrm{tr}\left(C_eP_eC_e^{\top}\right)+\mathrm{tr}\left(P_eM_eP_eM_e\right) \nonumber \\
    & = \mathrm{tr}(CPC^\top - 2CX\hat{C}^\top + \hat{C}\hat{P}\hat{C}^\top) + \mathrm{tr}(PMPM - 2X^\top M X \hat{M} + \hat{P}\hat{M}\hat{P}\hat{M}),
    \label{3.1.10}
\end{align}
where $X, Y, \hat{P}, \hat{Q}$ are solutions to the Sylvester equationss above.
The $H_2$-optimal model reduction problem is thus formulated as the following constrained optimization problem
\begin{equation}
    \min_{\substack{\hat{A}\in\mathbb{R}^{r\times r},\,\hat{A}\text{ is Hurwitz}\\\hat{B}\in\mathbb{R}^{r\times m},\,\hat{C}\in\mathbb{R}^{1\times r}\\\hat{M}\in \mathbb{S}_r}}J(\hat{A},\hat{B},\hat{C},\hat{M}).
    \label{3.1.11}
\end{equation}
The subsequent step involves reformulating problem \eqref{3.1.11} as an unconstrained optimization problem on a Riemannian manifold by endowing its constraint set with a suitable manifold structure.

\subsection{Optimization based on Grassmann manifolds}\label{section3.2}
This subsection introduces the Grassmann manifold, which allows the $H_2$-optimal MOR problem to be reformulated as an unconstrained optimization problem on this geometric space.
Relevant geometric definitions are then provided.
For a comprehensive treatment of matrix manifolds, see \cite{Absil2008,Boumal2023}.
\par
Within the Petrov-Galerkin projection framework, a reduced-order LQO model \eqref{2.2} is constructed using the following parameter matrix definitions:
\begin{equation}
    \hat{A} = W^\top AV,\quad \hat{B} = W^\top B,\quad \hat{C} = CV,\quad \hat{M} = V^\top MV.\nonumber
    \label{3.2.1}
\end{equation}
The projection matrices $W, V \in \mathbb{R}^{n \times r}$ should satisfy the biorthogonality constraint $W^\top V = I_r$, where $r$ is the dimension of the reduced-order system.
Substituting these definitions into the cost function \eqref{3.1.9}, the $H_2$-optimal model order reduction problem for the LQO system can be formulated as the following constrained optimization problem
\begin{equation}
    \min_{\substack{W,V \in \mathbb{R}^{n \times r} \\ W^\top V = I_r}} \{J_1(W,V) = J(W^\top AV, W^\top B, CV, V^\top MV)\}.
    \label{3.2.2}
\end{equation}
\par

The noncompact Stiefel manifold, denoted by $\mathbb{R}_*^{n \times r}$, represents the set of all $n \times r$ matrices with full column rank
\begin{equation*}
    \mathbb{R}_*^{n \times r} = \{U \in \mathbb{R}^{n \times r} \mid \mathrm{rank}(U) = r \}.
    \label{3.2.3}
\end{equation*}
For problem \eqref{3.2.2}, the biorthogonality constraint $W^\top V = I_r$ implicitly requires that the transformation matrices $W$ and $V$ have full column rank, since $r = \mathrm{rank}(I_r) = \mathrm{rank}(W^\top V) \leq \min(\mathrm{rank}(W^\top)$, $\mathrm{rank}(V)) \leq r$.
On this basis, problem \eqref{3.2.2} can be reformulated as an optimization problem over the product of noncompact Stiefel manifolds, subject to the biorthogonality constraint
\begin{equation}
    \min_{\substack{W,V \in \mathbb{R}_*^{n \times r} \\ W^\top V = I_r}} \{J_1(W,V) = J(W^\top AV, W^\top B, CV, V^\top MV)\}.
    \label{3.2.4}
\end{equation}
It is important to note that \eqref{3.2.4} is a constrained optimization problem defined on the manifold $\mathbb{R}_*^{n \times r} \times \mathbb{R}_*^{n \times r}$.
Due to the presence of the explicit biorthogonality constraint $W^\top V = I_r$, existing Riemannian optimization methods cannot be directly applied to solve this problem.
To this end, we will conduct a further analysis of problem \eqref{3.2.4} in the subsequent sections.
\par

In the Petrov-Galerkin projection framework, the cost function $J_1$ associated with problem \eqref{3.2.2} depends only on the subspaces spanned by the column vectors of the projection matrices $W$ and $V$.
This implies that the solution to this optimization problem is not a unique pair of matrices $(W,V)$, but rather consists of two sets of equivalence classes
\begin{equation}
    [W] := \{WQ \mid Q \in \mathrm{GL}_r\}, \quad [V] := \{VQ \mid Q \in \mathrm{GL}_r\},
    \label{3.2.5}
\end{equation}
where $\mathrm{GL}_r$ denotes the General Linear Group, which is the set of all non-singular $r \times r$ real matrices (i.e., $\mathrm{GL}_r = \{Q \in \mathbb{R}^{r \times r} \mid \det(Q) \neq 0\}$).
Each equivalence class $[W]$ (or $[V]$) corresponds to a unique $r$-dimensional subspace in $\mathbb{R}^n$.
In light of this observation, we proceed to introduce the Grassmann manifold.
\par

The Grassmann manifold, denoted by $\mathrm{Gr}(n,r)$, is defined as the set of all 
$r$-dimensional linear subspace in $\mathbb{R}^n$.
For further details on Grassmann manifolds, the reader is referred to references \cite{Absil2008} \cite{Boumal2023} \cite{Absil2004}.
On the noncompact Stiefel manifold $\mathbb{R}_*^{n \times r}$ (the set of all $n \times r$ real matrices with full column rank), we define an equivalence relation $\sim$ as follows: for any $U_1, U_2 \in \mathbb{R}_*^{n \times r}$,
\begin{equation*}
    U_1 \sim U_2 \iff \mathrm{span}(U_1) = \mathrm{span}(U_2),
    \label{3.2.6}
\end{equation*}
where $\mathrm{span}(\cdot)$ denotes the column space of a matrix.
Under this equivalence relation, and according to reference \cite{Absil2008}, the Grassmann manifold $\mathrm{Gr}(n,r)$ can be identified with the quotient manifold of $\mathbb{R}_*^{n \times r}$ by $\sim$. That is,
\begin{equation*}
    \mathrm{Gr}(n,r) = \mathbb{R}_*^{n \times r} / \sim.
    \label{3.2.7}
\end{equation*}
Each point on $\mathrm{Gr}(n,r)$ thus corresponds to an equivalence class of matrices in $\mathbb{R}_*^{n \times r}$ whose columns span the same $r$-dimensional subspace.
\par

The tangent space to the Grassmann manifold $\mathrm{Gr}(n,r)$ at a point represented by the equivalence class $[U]$ (where $U \in \mathbb{R}_*^{n \times r}$ is a full-rank matrix whose columns span the $r$-dimensional subspace) is denoted by $T_{[U]}\mathrm{Gr}(n,r)$.
According to reference \cite{Absil2008}, this tangent space can be identified with a set of $n \times r$ matrices as follows
\begin{equation*}
    T_{[U]}\mathrm{Gr}(n,r) = \{\xi \in \mathbb{R}^{n \times r} \mid \xi^\top U = \mathbf{0}_{r \times r} \}.
    \label{3.2.8}
\end{equation*}
The orthogonal projection operator $P_U$ onto this representation of the tangent space $T_{[U]}\mathrm{Gr}(n,r)$ is given by \cite{Absil2008} for any $Z \in \mathbb{R}^{n \times r}$ as
\begin{equation}
    P_U(Z) = (I - U(U^\top U)^{-1}U^\top)Z.
    \label{3.2.9}
\end{equation}
The Grassmann manifold $\mathrm{Gr}(n,r)$ is endowed with the standard Riemannian metric, which, using the representative matrix $U$, is defined for any $\xi, \eta \in T_{[U]}\mathrm{Gr}(n,r)$ by
\begin{equation*}
    \langle \xi, \eta \rangle_{[U]} = \mathrm{tr}\left((U^\top U)^{-1}\xi^\top \eta\right).
    \label{3.2.10}
\end{equation*}
With this metric, $\mathrm{Gr}(n,r)$ becomes a Riemannian quotient manifold of the noncompact Stiefel manifold $\mathbb{R}_*^{n \times r}$.
\par

Building upon the introduction of the Grassmann manifold $\mathrm{Gr}(n,r)$, the equivalence classes $[W]$ and $[V]$ defined in \eqref{3.2.5} constitute two elements (i.e., $r$-dimensional subspaces) in $\mathrm{Gr}(n,r)$.
On this basis, problem \eqref{3.2.4} can be transformed into an optimization problem over pairs of elements on the Grassmann manifold
\begin{equation}
    \min_{\substack{[W],[V] \in \mathrm{Gr}(n,r) \\ \text{s.t. } \exists W \in [W], V \in [V] : W^\top V=I_r}} J_1(W,V).
    \label{3.2.11}
\end{equation}
Problem \eqref{3.2.11} optimizes over the equivalence classes $[W]$ and $[V]$ on $\mathrm{Gr}(n,r)$, which equivalently achieves the optimization with respect to the projection matrices $W$ and $V$ as posed in problem \eqref{3.2.4}.
The advantage of analyzing optimization problem \eqref{3.2.11} rather than 
\eqref{3.2.4} lies in the fact that for any pair of subspaces $([W], [V])$ that 
admit biorthogonal bases, we can always select representative matrices $W \in [W]$ 
and $V \in [V]$ that satisfy the constraint $W^\top V = I_r$ for the evaluation of 
$J_1(W,V)$,which can potentially simplify the handling of the biorthogonality 
constraint in the optimization process by focusing the search on the geometric space 
of subspace.
\par

\subsection{Bivariable alternating optimization algorithm}\label{section3.3}
In this subsection, we design optimization algorithms on the Grassmann manifold to solve problem \eqref{3.2.11}.
As a preliminary step, we first compute the partial derivatives of the cost function $J_1(W,V)$ with respect to the elements of $W$ and $V$.
Based on these derivatives, we will then derive the expression for the Riemannian gradient of the cost function on the Grassmann manifold.
\par

We first introduce a proposition that reveals the trace properties of matrices satisfying certain Sylvester equations.
\newtheorem{proposition}{Proposition}
\begin{proposition}
\label{prop1}
If $P$ and $Q$ satisfy $AP+PB+X=0$ and $A^\top Q+QB^\top+Y=0$, then it holds that $\mathrm{tr}(Y^\top P)=\mathrm{tr}(X^\top Q)$.
\end{proposition}
\par

For the two-variable matrix function $J_1(W,V)$, where $W, V \in \mathbb{R}^{n \times r}$, we define the $n \times r$ matrices $\nabla_W J_1(W,V)$ (denoted as $J_{1,W}$) and $\nabla_V J_1(W,V)$ (denoted as $J_{1,V}$) as the partial derivatives (gradients) of $J_1(W,V)$ with respect to the matrix variables $W$ and $V$, respectively.
The $(i,j)$-th element of these gradient matrices satisfies
\begin{equation*}
    (\nabla_W J_1)_{ij} = \frac{\partial J_1(W,V)}{\partial W_{ij}} \quad \text{and} \quad (\nabla_V J_1)_{ij} = \frac{\partial J_1(W,V)}{\partial V_{ij}},
    \label{3.3.0}
\end{equation*}
where $W_{ij}$ and $V_{ij}$ represent the $(i,j)$-th elements of the matrices $W$ and $V$, respectively, for $i=1,2,\dots,n$ and $j=1,2,\dots,r$.

Let $\Delta_W \in \mathbb{R}^{n \times r}$ and $\Delta_V \in \mathbb{R}^{n \times r}$ be infinitesimal increments (perturbations) to the variables $W$ and $V$, respectively.
Then, according to reference \cite{Zeng2015}, the first-order Taylor expansions of $J_1(W,V)$ are given by
\begin{align}
    J_1(W + \Delta_W, V) &= J_1(W,V) + \mathrm{tr}(\Delta_W^\top \nabla_W J_1(W,V)) + o(\|\Delta_W\|_F), \label{eq:taylor_W} \\
    J_1(W, V + \Delta_V) &= J_1(W,V) + \mathrm{tr}(\Delta_V^\top \nabla_V J_1(W,V)) + o(\|\Delta_V\|_F), \label{eq:taylor_V}
\end{align}
where $\mathrm{tr}(\cdot)$ denotes the trace operator and $o(\|\cdot\|_F)$ represents higher-order terms with respect to the Frobenius norm.
\par

Subsequently, building upon the Taylor expansions for $J_1(W+\Delta_W, V)$ and $J_1(W, V+\Delta_V)$ previously established in \eqref{eq:taylor_W} and \eqref{eq:taylor_V}, we proceed to derive the partial derivatives of the function $J_1(W,V)$ with respect to its matrix variables $W$ and $V$.
With the aid of \cref{prop1}, the following theorem provides the explicit expressions for these partial derivative matrices, $J_{1,W}$ and $J_{1,V}$.
\par

\newtheorem{theorem}{Theorem}
\begin{theorem}
\label{lemma1}
Consider an asymptotically stable LQO system \eqref{2.1}. The partial derivatives of 
the cost function $J_1(W,V)$ associated with problem \eqref{3.2.11} with respect to 
$W$ and $V$, denoted $J_{1,W}$ and $J_{1,V}$ respectively, are given by
\begin{equation*}
	\begin{split}
    J_{1,W} &= 2\left(AV(X^\top K + \hat{P}L) + BB^\top(K+WL)\right),\label{3.3.1} \\
    J_{1,V} &= 2\left(A^\top W(K^\top X + L\hat{P}) + C^\top C(V\hat{P}-X) + 
    2MV(\hat{P}\hat{M}\hat{P} - X^\top M X)\right).
    \end{split}
\end{equation*}
where $X, \hat{P}$ are the solutions to \eqref{3.1.5}, \eqref{3.1.6}, and $K, L$ are the solutions to the following Sylvester equations
\begin{align}
    A^\top K+K\hat{A}-C^\top\hat{C}-2MX\hat{M}=&0,\label{3.3.3}\\
	\hat{A}^\top L+L\hat{A}+\hat{C}^\top\hat{C}+2\hat{M}\hat{P}\hat{M}=&0.\label{3.3.4}
\end{align}
\end{theorem}

\begin{proof}
Since the derivation of $J_{1,W}$ is similar to that of $J_{1,V}$, we only present the derivation for $J_{1,W}$ below. The derivation for $J_{1,V}$ follows analogously.

To derive $J_{1,W}$, consider the cost function
\begin{equation*}
    J_1(W,V) = \mathrm{tr}(B^\top Q B + 2B^\top Y B + \hat{B}^\top \hat{Q} \hat{B}).
    \label{3.3.5}
\end{equation*}

Let the first-order differential of $J_1(W,V)$ with respect to $\Delta_W$ be denoted by $\Delta J_1$. Then,
\begin{equation}
    \Delta J_1 = 2\mathrm{tr}(\Delta_W^\top(BB^\top(Y+W\hat{Q}))) + 2\mathrm{tr}(\hat{B}B^\top\Delta_Y) + \mathrm{tr}(\hat{B}\hat{B}^\top\Delta_{\hat{Q}}), \label{3.3.6}
\end{equation}
where $\Delta_Y$ and $\Delta_{\hat{Q}}$ respectively satisfy the equations
\begin{align}
    A^\top \Delta_Y + \Delta_Y \hat{A} + Y \Delta_W^\top AV - M\Delta_X \hat{M} &= 0, \label{3.3.7} \\
    \hat{A}^\top \Delta_{\hat{Q}} + \Delta_{\hat{Q}} \hat{A} + V^\top A^\top \Delta_W \hat{Q} + \hat{Q} \Delta_W^\top AV + \hat{M}\Delta_{\hat{P}}\hat{M} &= 0, \label{3.3.8}
\end{align}
where $\Delta_X$ and $\Delta_{\hat{P}}$ respectively satisfy the equations
\begin{equation}
    A\Delta_X + \Delta_X \hat{A}^\top + XV^\top A^\top \Delta_W + BB^\top \Delta_W = 0, \label{3.3.9}
\end{equation}
\begin{equation}
    \hat{A}\Delta_{\hat{P}} + \Delta_{\hat{P}}\hat{A}^\top + \Delta_W^\top AV\hat{P} + \hat{P}V^\top A^\top\Delta_W + \Delta_W^\top B\hat{B}^\top + \hat{B}B^\top\Delta_W = 0. \label{3.3.10}
\end{equation}
Note that symbols such as $\Delta_Y$, $\Delta_{\hat{Q}}$, $\Delta_X$, and $\Delta_{\hat{P}}$ represent the first-order differentials of the matrices $Y$, $\hat{Q}$, $X$, and $\hat{P}$, respectively, with respect to the perturbation $\Delta_W$.
Using \eqref{3.1.5}, \eqref{3.3.7}, and \cref{prop1}, we obtain
\begin{equation}\label{3.3.11}
\begin{split}
    \mathrm{tr}(\hat{B}B^\top \Delta_Y) &= \mathrm{tr}(X^\top(Y\Delta_W^\top AV - M\Delta_X\hat{M})) \\
    &= \mathrm{tr}(\Delta_W^\top AVX^\top Y) - \mathrm{tr}(\hat{M}X^\top M\Delta_X).
\end{split}
\end{equation}
Construct the Sylvester equation
	\begin{equation}\label{3.3.12}
		A^\top R_1+R_1\hat{A}+MX\hat{M}=0.
	\end{equation}
Then, from \eqref{3.3.12} and \eqref{3.3.9}, we can obtain
\begin{equation}\label{3.3.13}
\begin{split}
    \mathrm{tr}(\hat{M}X^\top M\Delta_X) &= \mathrm{tr}(R_1^\top(XV^\top A^\top\Delta_W + BB^\top\Delta_W)) \\
    &= \mathrm{tr}(\Delta_W^\top(AVX^\top R_1 + BB^\top R_1)).
\end{split}
\end{equation}
Substituting \eqref{3.3.13} into \eqref{3.3.11} and simplifying yields
\begin{equation}
    \mathrm{tr}(\hat{B}B^\top \Delta_Y) = \mathrm{tr}(\Delta_W^\top (AVX^\top(Y-R_1) - BB^\top R_1)). \label{3.3.14}
\end{equation}
Similarly, using \eqref{3.1.6}, \eqref{3.3.8}, and \cref{prop1}, we obtain
\begin{equation}\label{3.3.15}
\begin{split}
    \mathrm{tr}(\hat{B}\hat{B}^\top \Delta_{\hat{Q}}) &= \mathrm{tr}(\hat{P}(V^\top A^\top \Delta_W \hat{Q} + \hat{Q}\Delta_W^\top AV + \hat{M}\Delta_{\hat{P}}\hat{M})) \\
    &= 2\mathrm{tr}(\Delta_W^\top AV\hat{P}\hat{Q}) + \mathrm{tr}(\hat{M}\hat{P}\hat{M}\Delta_{\hat{P}}).
\end{split}
\end{equation}
Construct the Sylvester equation
\begin{equation}\label{3.3.16}
	\hat{A}^\top R_2+R_2\hat{A}+\hat{M}\hat{P}\hat{M}=0.
\end{equation}
Using \eqref{3.3.16}, \eqref{3.3.10}, and \cref{prop1}, we get
\begin{equation}\label{3.3.17}
\begin{split}
    \mathrm{tr}(\hat{M}\hat{P}\hat{M}\Delta_{\hat{P}}) &= \mathrm{tr}(R_2(\Delta_W^\top AV\hat{P} + \hat{P}V^\top A^\top\Delta_W + \Delta_W^\top B\hat{B}^\top + \hat{B}B^\top\Delta_W)) \\
    &= 2\mathrm{tr}(\Delta_W^\top(AV\hat{P}R_2 + BB^\top R_2)).
\end{split}
\end{equation}
Substituting \eqref{3.3.17} into \eqref{3.3.15} yields
\begin{equation}
    \mathrm{tr}(\hat{B}\hat{B}^\top \Delta_{\hat{Q}}) = 2\mathrm{tr}(\Delta_W^\top(AV\hat{P}(\hat{Q} + R_2) + BB^\top R_2)). \label{3.3.18}
\end{equation}
Substituting \eqref{3.3.14} and \eqref{3.3.18} into \eqref{3.3.6} and let $K=Y-R_1$, $L=\hat{Q}+R_2$, and simplifying , we obtain
\begin{equation*}
    \Delta J_1(W,V) = 2\mathrm{tr}(\Delta_W^\top(AV(X^\top K + \hat{P}L) + BB^\top(K + WL))). \label{3.3.19}
\end{equation*}
According to the definition $\Delta J_1(W,V) = \mathrm{tr}(\Delta_W^\top J_{1,W})$, the theorem is proven.
\end{proof}
\par

\cref{lemma1} provides the partial derivatives (Euclidean gradients) of the cost function $J_1(W,V)$ with respect to the variables $W$ and $V$.
Building upon this, and according to \eqref{3.2.9}, the Riemannian gradients of the function $J_1(W,V)$ at points represented by the equivalence classes $[W]$ and $[V]$ on the Grassmann manifold are given respectively by
\begin{equation*}
	\begin{split}
    \nabla_{J_{1,W}}(W,V) &= (I - W(W^\top W)^{-1}W^\top)J_{1,W}W^\top W, \label{3.3.20} \\
    \nabla_{J_{1,V}}(W,V) &= (I - V(V^\top V)^{-1}V^\top)J_{1,V}V^\top V. 
    \end{split}
\end{equation*}
\par

To solve problem \eqref{3.2.11}, we now develop a bivariable alternating optimization algorithm on the Grassmann manifold $\mathrm{Gr}(n,r)$.
The core strategy of this algorithm involves decomposing the main optimization problem into two subproblems within each iteration.
A locally optimal solution is then progressively sought by alternately optimizing one variable while holding the other constant.
The detailed algorithmic steps are presented below.
\par

Let $W_k$ and $V_k$ be the known projection matrices at the $k$-th iteration, for $k = 0, 1, \dots$.
Within the Petrov-Galerkin projection framework, the coefficient matrices of the reduced-order system at the $k$-th iteration are given by
\[
\hat{A}_k = W_k^\top A V_k, \quad \hat{B}_k = W_k^\top B, \quad \hat{C}_k = C V_k, \quad \hat{M}_k = V_k^\top M V_k.
\]
Subsequently, we detail the construction process for the projection matrices $W_{k+1}$ and $V_{k+1}$ at the $(k+1)$-th iteration.
\par

To determine the projection matrix $V_{k+1}$ in the $(k+1)$-th iteration, we first fix the projection matrix $W_k$.
The optimization problem \eqref{3.2.11} then transforms into a nonlinear minimization problem with respect to $V$
\begin{equation}
    \min_{\substack{[V] \in \mathrm{Gr}(n,r) \\ W_k^\top V = I_r}} \{J_1(W_k, V) := J(W_k^\top A V, W_k^\top B, C V, V^\top M V)\}. \label{3.3.22}
\end{equation}
Recall the expression for the cost function $J(\hat{A}, \hat{B}, \hat{C}, \hat{M})$ from \eqref{3.1.10}
\begin{align*}
    J(\hat{A}, \hat{B}, \hat{C}, \hat{M}) ={}& \mathrm{tr}(CPC^\top - 2CX\hat{C}^\top + \hat{C}\hat{P}\hat{C}^\top) \\
    & + \mathrm{tr}(PMPM - 2X^\top M X \hat{M} + \hat{P}\hat{M}\hat{P}\hat{M}),
\end{align*}
The matrices $X$ and $\hat{P}$ in this expression depend nonlinearly on $W_k$ and $V$, making problem \eqref{3.3.22} challenging to solve directly.
To address this, we define an approximation of the cost function $J_1(W_k, V)$, denoted by $\mathcal{F}(V)$
\begin{equation*}
\begin{split}
    \mathcal{F}(V) ={}& \mathrm{tr}(CPC^\top - 2CX_k\hat{C}^\top + \hat{C}\hat{P}_k\hat{C}^\top) \\
    & + \mathrm{tr}(PMPM - 2X_k^\top M X_k \hat{M} + \hat{P}_k\tilde{M}\hat{P}_k\tilde{M}),
\end{split} \label{3.3.23}
\end{equation*}
where $X_k$ and $\hat{P}_k$ are solutions to the following Sylvester equations, respectively
\begin{equation*}
	\begin{split}
    AX_k + X_k\hat{A}_k^\top + B\hat{B}_k^\top &= 0, \label{3.3.24} \\
    \hat{A}_k\hat{P}_k + \hat{P}_k\hat{A}_k^\top + \hat{B}_k\hat{B}_k^\top &= 0. 
    \end{split}
\end{equation*}
By solving the problem
\begin{equation}
    \min_{\substack{[V] \in \mathrm{Gr}(n,r) \\ W_k^\top V = I_r}} \mathcal{F}(V), \label{3.3.26}
\end{equation}
we obtain an approximate solution to the original problem \eqref{3.3.22}.
\par

The approximate problem \eqref{3.3.26} is solved using a line search method.
Analogously, the derivative (Euclidean gradient) of the function $\mathcal{F}(V)$ with respect to $V$, denoted $\mathcal{F}_V$, can be derived as
\[
\mathcal{F}_V = 2 \left( C^\top C(V\hat{P}_k - X_k) + 2MV \left( (V\hat{P}_k - X_k)^\top M (V\hat{P}_k - X_k) \right) \right).
\]
Evidently, if a matrix $\tilde{V}_{k+1}$ satisfies the condition $\tilde{V}_{k+1}\hat{P}_k = X_k$, then the gradient $\mathcal{F}_V$ evaluated at $\tilde{V}_{k+1}$ (denoted $\mathcal{F}_{\tilde{V}_{k+1}}$) is zero.
If $\hat{P}_k$ is non-singular, this condition yields $\tilde{V}_{k+1} = X_k \hat{P}_k^{-1}$.
The search direction $S_k$ is then constructed by orthogonally projecting the difference $\tilde{V}_{k+1} - V_k$ onto the tangent space to $\mathrm{Gr}(n,r)$ at the point represented by $V_k$
\begin{equation}
\begin{split}
    S_k &= P_{[V_k]}(\tilde{V}_{k+1} - V_k) \\
        &= \tilde{V}_{k+1} - V_k(V_k^\top V_k)^{-1}V_k^\top \tilde{V}_{k+1},
\end{split} \label{3.3.27}
\end{equation}
where $P_{[V_k]}$ is the projection operator defined in \eqref{3.2.9}.
\par

Once the search direction $S_k$ is determined, the Armijo rule is employed to compute the iterative step length $t_k$.
To this end, we define a candidate matrix $\mathcal{V}_k(t)$ along the search direction, dependent on a step length $t \ge 0$
\[
\mathcal{V}_k(t) = (V_k + tS_k)(W_k^\top(V_k + tS_k))^{-1}, \quad t \ge 0.
\]
It is clear that $\mathcal{V}_k(t)$, as constructed, satisfies the biorthogonality condition $W_k^\top \mathcal{V}_k(t) = I_r$ with $W_k$.
Using $\mathcal{V}_k(t)$, the Armijo step length $t_k = \omega_1^{m_k} \gamma_1$ is determined via a backtracking strategy, where $\omega_1 \in (0,1)$ is a contraction factor, $\gamma_1 > 0$ is an initial trial step length, and $m_k$ is the smallest non-negative integer such that the following inequalities are satisfied
\begin{align}
    J_1(W_k, V_k) - J_1(W_k, \mathcal{V}_k(\omega_1^{m_k} \gamma_1)) &\ge \alpha_1 
    \omega_1^{m_k} \gamma \langle S_k, \mathcal{F}_{V_k} \rangle_{[V_k]}. 
    \label{3.3.28a}
\end{align}
To ensure the stability of the reduced-order system, $m_k$ is also chosen such that the real part of the trace of the product $W_k^\top A \mathcal{V}_k(\omega_1^{m_k} \gamma_1)$ is negative.
The parameter $\alpha_1 \in (0,1)$ ensures sufficient decrease in the cost function.
Once a suitable $t_k$ satisfying these conditions is found, the projection matrix for the $(k+1)$-th iteration is set as $V_{k+1} = \mathcal{V}_k(t_k)$.
\par

Analogous to the construction of $V_{k+1}$, we now detail the procedure for obtaining the projection matrix $W_{k+1}$.
Assuming $W_k$ (current iterate for $W$) and $V_{k+1}$ (newly computed iterate for $V$) are known, we fix $V_{k+1}$.
Problem \eqref{3.2.11} then becomes an optimization problem with respect to $W$
\begin{equation}
    \min_{\substack{[W] \in \mathrm{Gr}(n,r) \\ W^\top V_{k+1} = I_r}} \{J_1(W, V_{k+1}) := J(W^\top A V_{k+1}, W^\top B, C V_{k+1}, V_{k+1}^\top M V_{k+1})\}. \label{3.3.29}
\end{equation}
To derive an approximate solution to problem \eqref{3.3.29}, we construct the following approximate cost function $\mathcal{G}(W)$, based on \eqref{3.1.9}
\begin{equation*}
    \mathcal{G}(W) = \mathrm{tr}(B^\top Q B + 2B^\top Y_k \hat{B} + \hat{B}^\top \hat{Q}_k \hat{B}), \label{3.3.30}
\end{equation*}
where $\hat{B} = W^\top B$. The matrices $Y_k$ and $\hat{Q}_k$ correspond to the reduced-order system obtained using projection matrices $(W_k, V_{k+1})$ and are computed by solving equations \eqref{3.1.7} and \eqref{3.1.8}, respectively.
An updated $W_{k+1}$ is then found by solving the optimization problem
\begin{equation*}
    \min_{\substack{[W] \in \mathrm{Gr}(n,r) \\ W^\top V_{k+1} = I_r}} \mathcal{G}(W). \label{3.3.31}
\end{equation*}
This yields an approximate solution to the original minimization problem \eqref{3.3.29} for $W$.
\par

The derivative (Euclidean gradient) of the function $\mathcal{G}(W)$ with respect to $W$, denoted $\mathcal{G}_W$, can be derived as
\[
\mathcal{G}_W = 2BB^\top(Y_k + W\hat{Q}_k).
\]
Evidently, if a matrix $\tilde{W}_{k+1}$ satisfies the condition $Y_k + \tilde{W}_{k+1}\hat{Q}_k = 0$, then the gradient $\mathcal{G}_W$ evaluated at $\tilde{W}_{k+1}$ (denoted $\mathcal{G}_{\tilde{W}_{k+1}}$) is zero.
If $\hat{Q}_k$ is non-singular, this condition yields $\tilde{W}_{k+1} = -Y_k(\hat{Q}_k)^{-1}$.
The search direction $S_{W_k}$ is then obtained by orthogonally projecting the difference $\tilde{W}_{k+1} - W_k$ onto the tangent space $T_{[W_k]}\mathrm{Gr}(n,r)$ at the point represented by $W_k$
\begin{equation}
\begin{split}
    S_{W_k} &= P_{[W_k]}(\tilde{W}_{k+1} - W_k) \\
            &= \tilde{W}_{k+1} - W_k(W_k^\top W_k)^{-1}W_k^\top \tilde{W}_{k+1},
\end{split} \label{3.3.32}
\end{equation}
where $P_{[W_k]}$ is the projection operator (defined analogously to $P_{[V_k]}$ in \eqref{3.2.9}).
\par

Once the search direction $S_{W_k}$ from \eqref{3.3.32} is determined, the Armijo rule is applied to compute the step length $t_{W_k}$.
To this end, we define a candidate matrix $\mathcal{W}_k(t)$ dependent on a step length $t \ge 0$
\[
\mathcal{W}_k(t) = (W_k + tS_{W_k})(V_{k+1}^\top(W_k + tS_{W_k}))^{-1}, \quad t \ge 0.
\]
The Armijo step length $t_{W_k} = \omega_2^{n_k} \gamma_2$ is then found via a backtracking line search.
Here, $\omega_2 \in (0,1)$ is a contraction factor, $\gamma_2 > 0$ is an initial trial step length, and $n_k$ is the smallest non-negative integer for which the following conditions are met
\begin{align}
    J_1(W_k, V_{k+1}) - J_1(\mathcal{W}_k(\omega_2^{n_k} \gamma_2), V_{k+1}) &\ge \alpha_2 \omega_2^{n_k} \gamma_2 \langle S_{W_k}, \mathcal{G}_{W_k} \rangle_{[W_k]}, \label{3.3.33a}
\end{align}
Similarly, $n_k$ is also chosen such that the real part of the trace of the product 
$\mathcal{W}_k(\omega_2^{n_k} \gamma_2)^\top A V_{k+1}$ is negative.
The parameter $\alpha_2 \in (0,1)$ ensures sufficient decrease.
Upon finding a suitable $t_{W_k}$, the projection matrix for the $(k+1)$-th iteration is updated as $W_{k+1} = \mathcal{W}_k(t_{W_k})$.
The procedure described completes the update from iteration $k$ to $k+1$ for the bivariable alternating optimization algorithm solving problem \eqref{3.2.11} on the Grassmann manifold $\mathrm{Gr}(n,r)$.
With $W_{k+1}$ and $V_{k+1}$ thus determined, the coefficient matrices of the reduced-order model at the $(k+1)$-th iteration are given by
\[
(\hat{A}_{k+1}, \hat{B}_{k+1}, \hat{C}_{k+1}, \hat{M}_{k+1}) = (W_{k+1}^\top A V_{k+1}, W_{k+1}^\top B, C V_{k+1}, V_{k+1}^\top M V_{k+1}).
\]
\par

The specific steps of the bivariable alternating optimization algorithm for solving the Riemannian optimization problem \eqref{3.2.11} are presented in Algorithm 4.1.

\begin{algorithm}[H]
\caption{Bivariable Alternating Optimization Algorithm for $H_2$-Optimal Model Reduction of LQO Systems on the Grassmann Manifold.}
\label{alg1}
\begin{algorithmic}
    \STATE \textbf{Input:} LQO system $\Sigma$.
    \STATE \textbf{Output:} Reduced-order LQO system $\hat{\Sigma}$ with matrices $(\hat{A}, \hat{B}, \hat{C}, \hat{M})$.
    \STATE \textbf{Initialize:} Choose initial matrices $W_0, V_0 \in \mathbb{R}^{n \times r}$ such that $W_0^\top V_0 = I_r$.
    \FOR{$k = 0, 1, \dots$ (until convergence)}
        \STATE Compute search direction $S_{V_k}$ using \eqref{3.3.27}.
        \STATE Compute step length $t_{V_k}$ via Armijo rule \eqref{3.3.28a} using $\mathcal{V}_k(t)$.
        \STATE Update $V_{k+1} \leftarrow \mathcal{V}_k(t_{V_k})$.
        \STATE Compute search direction $S_{W_k}$ using \eqref{3.3.32}.
        \STATE Compute step length $t_{W_k}$ via Armijo rule \eqref{3.3.33a} using $\mathcal{W}_k(t)$.
        \STATE Update $W_{k+1} \leftarrow \mathcal{W}_k(t_{W_k})$.
        \STATE Construct $(\hat{A}_{k+1}, \hat{B}_{k+1}, \hat{C}_{k+1}, \hat{M}_{k+1}) \leftarrow (W_{k+1}^\top A V_{k+1}, W_{k+1}^\top B, C V_{k+1}, V_{k+1}^\top M V_{k+1})$.
    \ENDFOR
    \STATE \textbf{Return:} $\hat{A} = W^\top A V, \hat{B} = W^\top B, \hat{C} = C V, \hat{M} = V^\top M V$ (using converged $W, V$).
\end{algorithmic}
\end{algorithm}

\subsection{Riemannian optimization MOR with the guarantee of 
stability}\label{section3.4}
In Algorithm 4.1, the stability of the reduced-order LQO system is ensured through a backtracking line search, which selects a step length satisfying specific stability-enforcing conditions.
Specifically, within each iteration, if a trial step length resulted in a reduced-order system that did not meet the asymptotic stability condition, the step length was adjusted via backtracking.
A key limitation of this approach is its inability to guarantee that a stability-preserving step length can be found at every iteration.
To address this limitation, this section introduces a novel model reduction algorithm for LQO systems within the Petrov-Galerkin projection framework.
By imposing specific constraints on the projection matrices, this new algorithm is 
designed to theoretically ensure the stability of the resulting reduced models 
naturally.
\par

We first revisit the $H_2$-optimal model order reduction (MOR) problem for Linear Quadratic Output (LQO) systems under the Petrov-Galerkin projection framework, as formulated in \eqref{3.2.2}
\[
\min_{\substack{W,V \in \mathbb{R}^{n \times r} \\ W^\top V = I_r}} \{J_1(W,V) := J(\hat{A}, \hat{B}, \hat{C}, \hat{M})\},
\]
where the cost function $J$ is given by
\[
J(\hat{A}, \hat{B}, \hat{C}, \hat{M}) = \mathrm{tr}(B^\top Q B + 2B^\top Y \hat{B} + \hat{B}^\top \hat{Q} \hat{B}).
\]
The reduced-order system matrices $(\hat{A}, \hat{B}, \hat{C}, \hat{M})$ are obtained via Petrov-Galerkin projection as
\[
\hat{A} = W^\top A V, \quad \hat{B} = W^\top B, \quad \hat{C} = C V, \quad \hat{M} = V^\top M V.
\]
The matrices $Y$ and $\hat{Q}$ are solutions to equations \eqref{3.1.7} and 
\eqref{3.1.8}, respectively.
Drawing from the discussion in references \cite{Gugercin2008,Cheng2023}, for the projection matrices $W$ and $V$ in problem \eqref{3.2.2}, we impose the constraint
\[
W^\top = (V^\top H V)^{-1}V^\top H,
\]
where $H$ is an arbitrary symmetric positive definite matrix satisfying the Lyapunov inequality $A^\top H + H A < 0$, with $A$ being the system matrix of the original LQO system.
It can be readily verified that this choice ensures the biorthogonality condition $W^\top V = I_r$.
Furthermore, the reduced-order system matrix $\hat{A}$ is given by
\[
\hat{A} = W^\top A V = (V^\top H V)^{-1}V^\top H A V,
\]
an asymptotically stable matrix \cite{Cheng2023}.
This construction thereby guarantees the asymptotic stability of the reduced-order system.
\par

By substituting the constraint $W^\top = (V^\top H V)^{-1}V^\top H$ into problem \eqref{3.2.2}, the optimization is transformed into a problem solely involving the variable $V$
\begin{equation}
    \min_{\mathrm{rank}(V)=r} \{J_2(V) := J(\hat{A}, \hat{B}, \hat{C}, \hat{M})\}, \label{3.4.1}
\end{equation}
where the coefficient matrices $(\hat{A}, \hat{B}, \hat{C}, \hat{M})$ of the reduced-order model are now computed as
\begin{equation}
\begin{alignedat}{2}
    \hat{A} &= (V^\top H V)^{-1}V^\top H A V, &\quad \hat{B} &= (V^\top H V)^{-1}V^\top H B,\\
    \hat{C} &= C V, &\quad \hat{M} &= V^\top M V.
\end{alignedat} \label{3.4.2}
\end{equation}
\par

Recall from Section 4.1 that the non-compact Stiefel manifold $\mathbb{R}_*^{n \times r}$ denotes the set of all $n \times r$ matrices with full column rank.
Consequently, problem \eqref{3.4.1}, subject to the constraint $\mathrm{rank}(V)=r$, can be reformulated as an optimization problem on the non-compact Stiefel manifold $\mathbb{R}_*^{n \times r}$
\begin{equation}
    \min_{V \in \mathbb{R}_*^{n \times r}} J_2(V). \label{3.4.3}
\end{equation}
\par

For problem \eqref{3.4.3}, the cost function $J_2(V)$ depends only on the column space spanned by $V$.
Considering this invariance, we can impose the orthogonality constraint $V^\top V = I_r$ on problem \eqref{3.4.3} without altering its optimal value.
Consequently, optimization problem \eqref{3.4.3} is equivalent to
\begin{equation}
    \min_{\substack{V \in \mathbb{R}^{n \times r} \\ V^\top V = I_r}} \{J_2(V) := J(\hat{A}, \hat{B}, \hat{C}, \hat{M})\}. \label{3.4.4}
\end{equation}
The constraint set $\{V \in \mathbb{R}^{n \times r} \mid V^\top V = I_r\}$ defines the Stiefel manifold $St(n,r)$.
Problem \eqref{3.4.4} can therefore be reformulated as a Riemannian optimization problem on $St(n,r)$
\begin{equation}
    \min_{V \in St(n,r)} J_2(V). \label{3.4.5}
\end{equation}
\par

Next, we propose to solve the Riemannian optimization problem \eqref{3.4.5} using a Dai-Yuan type conjugate gradient method.
This necessitates deriving the expression for the Riemannian gradient of $J_2(V)$ on $St(n,r)$.
The following theorem provides the Riemannian gradient of the cost function $J_2(V)$ for problem \eqref{3.4.5} on $St(n,r)$.
\par

\begin{theorem}
\label{lemma2}
Consider an asymptotically stable LQO system \eqref{2.1} and its reduced-order 
system \eqref{2.2}.
Assume the coefficient matrices of the reduced-order system satisfy \eqref{3.4.2}, where the projection matrix $V \in St(n,r)$.
Further, assume that the matrices $X, \hat{P}, K, L$ are solutions to equations \eqref{3.1.5}, \eqref{3.1.6}, \eqref{3.3.3}, and \eqref{3.3.4}, respectively.
Then, the Riemannian gradient of the cost function $J_2(V)$ for optimization problem \eqref{3.4.5} on the Stiefel manifold $St(n,r)$ is expressed as
\begin{equation}
    \mathrm{grad}J_2(V) = \mathrm{grad}\bar{J}_2(V) - \frac{1}{2}V(V^\top \mathrm{grad}\bar{J}_2(V) + (\mathrm{grad}\bar{J}_2(V))^\top V), \label{3.4.6}
\end{equation}
where $\bar{J}_2(V)$ denotes a smooth extension of the cost function $J_2(V)$ to the ambient space $\mathbb{R}^{n \times r}$, and $\mathrm{grad}\bar{J}_2(V)$ is the Euclidean gradient of $\bar{J}_2(V)$, computed by
\begin{align*}
    \mathrm{grad}\bar{J}_2(V) = 2(&H(I - VW^\top)\mathcal{F}_1(V)(V^\top HV)^{-1} - W\mathcal{F}_1(V)^\top W + \mathcal{F}_2(V) \\
    & + A^\top W(X^\top K + \hat{P}L)^\top + 2MV(\hat{P}\tilde{M}\hat{P} - X^\top M X)),
\end{align*}
Here, $W = HV(V^\top HV)^{-1}$, and the functions $\mathcal{F}_1(V)$ and $\mathcal{F}_2(V)$ are defined as
\begin{align*}
    \mathcal{F}_1(V) &= AV(X^\top K + \hat{P}L) + BB^\top(K+WL),\\
    \mathcal{F}_2(V) &= C^\top C(V\hat{P} - X).
\end{align*}
\end{theorem}
\par

\begin{proof}
From the constraint $W^\top = (V^\top H V)^{-1}V^\top H$, we can readily obtain
\begin{equation}
    V^\top H = (V^\top H V)W^\top. \label{3.4.7}
\end{equation}
Taking the directional derivative of both sides of equation \eqref{3.4.7} with respect to $V$ in an arbitrary direction $\xi \in \mathbb{R}^{n \times r}$ yields
\begin{equation*}
\begin{split}
    \xi^\top H &= \mathrm{D}((V^\top HV)W^\top)[\xi] \\
    &= \mathrm{D}(V^\top HV)[\xi] W^\top + (V^\top HV)W'^\top \\
    &= (\xi^\top HV + V^\top H\xi)W^\top + (V^\top HV)W'^\top,
\end{split} \label{3.4.8}
\end{equation*}
where $W' = \mathrm{D}W[\xi]$ denotes the directional derivative of $W$ (as a function of $V$) in the direction $\xi$.\\
Rearranging the terms gives an expression for $W'^\top$
\begin{equation*}
    W'^\top = (V^\top H V)^{-1} \xi^\top H (I - VW^\top) - W^\top \xi W^\top. \label{3.4.8+}
\end{equation*}
The directional derivative of the function $\bar{J}_2(V)$ in the direction $\xi$ is given by
\begin{equation}
    \mathrm{D}\bar{J}_2(V)[\xi] = 2\mathrm{tr}(W'^\top BB^\top(Y+W\hat{Q})) + 2\mathrm{tr}(\hat{B}^\top Y') + \mathrm{tr}(\hat{B}\hat{B}^\top \hat{Q}'), \label{3.4.9}
\end{equation}
where $Y'=\mathrm{D}Y[\xi]$ and $\hat{Q}'=\mathrm{D}\hat{Q}[\xi]$ are the directional derivatives of the matrices $Y$ and $\hat{Q}$ with respect to $V$ in the direction $\xi$, and they satisfy the equations
\begin{equation}
    A^\top Y' + Y' \hat{A} + YW'^\top AV + YW^\top A\xi - C^\top C\xi -MX'\hat{M} - MX\xi^\top M V - MXV^\top M \xi = 0, \label{3.4.10}
\end{equation}
and\\
\begin{equation}
\begin{split}
    &\hat{A}'^\top \hat{Q}' + \hat{Q}' \hat{A} + \xi^\top A^\top W \hat{Q} + V^\top A^\top W' \hat{Q} + \hat{Q}W'^\top AV \\
    &+ \hat{Q}W^\top A\xi + \xi^\top C^\top \hat{C} + \hat{C}^\top C\xi + \xi^\top M V \hat{M} + V^\top M\xi \hat{M} \\
    &+ \hat{M} P' \hat{M} + \hat{M}\hat{P}\xi^\top M V + \hat{M}\hat{P}V^\top M \xi = 0,
\end{split} \label{3.4.11}
\end{equation}
where $X'$, $P'$, and $\hat{A}'$ are also directional derivatives in the direction $\xi$.\\
Here, $X'=\mathrm{D}X[\xi]$ and $\hat{P}'=\mathrm{D}\hat{P}[\xi]$ are the directional derivatives of $X$ and $\hat{P}$ in the direction $\xi$, satisfying
\begin{equation}
    AX' + X'\hat{A}^\top + X\xi^\top A^\top W + XV^\top A^\top W' + BB^\top W' = 0, \label{3.4.12}
\end{equation}
and\\
\begin{equation}
\begin{split}
    &\hat{A}\hat{P}' + \hat{P}'\hat{A}^\top + W'^\top AV\hat{P} + W^\top A\xi\hat{P} + \hat{P}\xi^\top A^\top W \\
    &+ \hat{P}V^\top A^\top W' + W'^\top B\hat{B}^\top + \hat{B}B^\top W' = 0.
\end{split} \label{3.4.13}
\end{equation}
We now derive expressions for the terms $\mathrm{tr}(\hat{B}B^\top Y')$ and $\mathrm{tr}(\hat{B}\hat{B}^\top \hat{Q}')$ that appear in \eqref{3.4.9}.\\
From \eqref{3.1.5} and \eqref{3.4.10}, and with the aid of \cref{prop1}, it follows that
\begin{equation}
\begin{split}
    \mathrm{tr}(\hat{B}B^\top Y') &= \mathrm{tr}((B\hat{B}^\top)^\top Y') \\
    &= \mathrm{tr}(X^\top(YW'^\top AV + YW^\top A\xi - C^\top C\xi \\
    &\qquad -MX'\hat{M} - MX\xi^\top MV - MXV^\top M\xi)) \\
    &= \mathrm{tr}(W'^\top AVX^\top Y) + \mathrm{tr}(\xi^\top A^\top WY^\top X - C^\top CX \\
    &\qquad - 2MVX^\top MX)) - \mathrm{tr}(\hat{M}X^\top MX').
\end{split} \label{3.4.14}
\end{equation}
From \eqref{3.3.12}, \eqref{3.4.12}, and with the aid of \cref{prop1}, we obtain
\begin{equation}
\begin{split}
    \mathrm{tr}(\hat{M}X^\top M X') &= \mathrm{tr}((MX\hat{M})^\top X') \\
    &= \mathrm{tr}(R_1^\top(X\xi^\top A^\top W + XV^\top A^\top W' + BB^\top W')) \\
    &= \mathrm{tr}(W'^\top(AVX^\top R_1 + BB^\top R_1)) + \mathrm{tr}(\xi^\top A^\top W R_1 X).
\end{split} \label{3.4.15}
\end{equation}
Substituting the result from \eqref{3.4.15} into \eqref{3.4.14} and rearranging terms yields
\begin{equation}
\begin{split}
    \mathrm{tr}(\hat{B}B^\top Y') ={}& \mathrm{tr}(W'^\top(AVX^\top(Y-R_1) - BB^\top R_1)) \\
    & + \mathrm{tr}(\xi^\top(A^\top W(Y-R_1)X - C^\top CX - 2MVX^\top MX)).
\end{split} \label{3.4.16}
\end{equation}
Similarly, from \eqref{3.1.6}, \eqref{3.4.11}, and \cref{prop1}, we have
\begin{equation}
\begin{split}
    \mathrm{tr}(\hat{B}\hat{B}^\top \hat{Q}') ={}& \mathrm{tr}(\hat{P}(\xi^\top A^\top W\hat{Q} + V^\top A^\top W'\hat{Q} + \hat{Q}W'^\top AV \\
    & + \hat{Q}W^\top A\xi + \xi^\top C^\top \hat{C} + \hat{C}^\top C\xi + \xi^\top M V\hat{P}\hat{M} \\
    & + V^\top M\xi\hat{P}\hat{M} + \hat{M}\hat{P}'\hat{M} + \hat{M}\hat{P}\xi^\top MV + \hat{M}\hat{P}V^\top M\xi)) \\
    ={}& 2\mathrm{tr}(W'^\top AV\hat{P}\hat{Q}) + 2\mathrm{tr}(\xi^\top(A^\top W\hat{Q}\hat{P} + C^\top C\hat{P} \\
    & + 2MV\hat{P}\hat{M}\hat{P})) + \mathrm{tr}(\hat{M}\hat{P}\hat{M}P').
\end{split} \label{3.4.17}
\end{equation}
From \eqref{3.3.16} and \eqref{3.4.13}, and in conjunction with \cref{prop1}, it follows that
\begin{equation}
\begin{split}
    \mathrm{tr}(\hat{M}\hat{P}\hat{M}P') &= \mathrm{tr}(R_2(W'^\top AV\hat{P} + W^\top A\xi\hat{P} + \hat{P}\xi^\top A^\top W \\
    &\qquad + \hat{P}V^\top A^\top W' + W'^\top B\hat{B}^\top + \hat{B}B^\top W')) \\
    &= 2\mathrm{tr}(W'^\top(AV\hat{P}R_2 + BB^\top R_2)) + 2\mathrm{tr}(\xi^\top A^\top WR_2\hat{P}).
\end{split} \label{3.4.18}
\end{equation}
Substituting this result \eqref{3.4.18} into \eqref{3.4.17} gives
\begin{equation}
\begin{split}
    \mathrm{tr}(\hat{B}\hat{B}^\top \hat{Q}') ={}& 2\mathrm{tr}(W'^\top(AV\hat{P}(\hat{Q} + R_2) + \hat{B}B^\top R_2)) \\
    & + 2\mathrm{tr}(\xi^\top(A^\top W(\hat{Q} + R_2)\hat{P} + C^\top C\hat{P} + 2MV\hat{P}\hat{M}\hat{P})).
\end{split} \label{3.4.19}
\end{equation}
Finally, substituting the expressions from \eqref{3.4.16} and \eqref{3.4.19} into the directional derivative formula \eqref{3.4.9} and collecting terms, we arrive at
\begin{equation}
\begin{split}
    \mathrm{D}\bar{J}_2(V)[\xi] ={}& 2\mathrm{tr}(W'^\top(AV(X^\top(Y-R_1) + \hat{P}(\hat{Q}+R_2)) + BB^\top(Y-R_1) \\
    & + B\hat{B}^\top(\hat{Q}+R_2))) + 2\mathrm{tr}(\xi^\top(A^\top W((Y-R_1)X+(\hat{Q}+R_2)\hat{P}) \\
    & + C^\top C(V\hat{P}-X) + 2MV(\hat{P}\hat{M}\hat{P} - X^\top MX))).
\end{split} \label{3.4.20}
\end{equation}
Let $K = Y - R_1$ and $L = \hat{Q} + R_2$. The matrices $K$ and $L$ can be shown to satisfy equations \eqref{3.3.3} and \eqref{3.3.4}, respectively.\\
Based on these definitions, equation \eqref{3.4.20} can be rewritten as
\begin{equation*}
\begin{split}
    \mathrm{D}\bar{J}_2(V)[\xi] ={}& 2\mathrm{tr}(W'^\top(AV(X^\top K + \hat{P}L) + BB^\top K + B\hat{B}^\top L)) \\
    & + 2\mathrm{tr}(\xi^\top(A^\top W(X^\top K + \hat{P}L)^\top + C^\top C(V\hat{P}-X) \\
    & + 2MV(\hat{P}\hat{M}\hat{P} - X^\top MX))).
\end{split} \label{3.4.21}
\end{equation*}
The directional derivative $\mathrm{D}\bar{J}_2(V)[\xi]$ can thus be consolidated into the following form
\begin{equation}
\begin{split}
    \mathrm{D}\bar{J}_2(V)[\xi] = 2\mathrm{tr}(\xi^\top(&H(I - VW^\top)\mathcal{F}_1(V)(V^\top HV)^{-1} - W\mathcal{F}_1(V)^\top W \\
    & + \mathcal{F}_2(V) + A^\top W(X^\top K + \hat{P}L)^\top + 2MV(\hat{P}\tilde{M}\hat{P} - X^\top MX)) ),
\end{split} \label{3.4.22}
\end{equation}
where\\
\begin{align*}
    \mathcal{F}_1(V) &= AV(X^\top K + \hat{P}L) + BB^\top(K+WL),\\
    \mathcal{F}_2(V) &= C^\top C(V\hat{P} - X).
\end{align*}
From \eqref{3.4.22}, by the definition of the Euclidean gradient, which satisfies $\mathrm{D}\bar{J}_2(V)[\xi] = \mathrm{tr}(\xi^\top \mathrm{grad}\bar{J}_2(V))$, we can identify the expression for the Euclidean gradient of $\bar{J}_2(V)$ as
\begin{equation*}
\begin{split}
    \mathrm{grad}\bar{J}_2(V) = 2(&H(I - VW^\top)\mathcal{F}_1(V)(V^\top HV)^{-1} - W\mathcal{F}_1(V)^\top W + \mathcal{F}_2(V) \\
    & + A^\top W(X^\top K + \hat{P}L)^\top + 2MV(\hat{P}\tilde{M}\hat{P} - X^\top MX)).
\end{split} \label{3.4.23}
\end{equation*}
With this expression for the Euclidean gradient, and in conjunction with the orthogonal projection onto the tangent space $T_{V}\mathrm{St}(n,r)$
\begin{equation*}
	P_V(D)=D-\frac12V(V^\top D+D^\top V),\quad D\in\mathbb{R}^{n\times r}.
\end{equation*}
the theorem is proven.
\end{proof}
\par

\cref{lemma2} provides the Riemannian gradient of the cost function for problem \eqref{3.4.5} on the Stiefel manifold $St(n,r)$.
Building on this, we can solve problem \eqref{3.4.5} using a Dai-Yuan type Riemannian conjugate gradient method.
We assume the iterate $V_k$ and search direction $\eta_k$ at the $k$-th step are known.
The subsequent iterate $V_{k+1}$ is then computed as follows
\begin{equation}
    V_{k+1} = \mathcal{R}_{V_k}(t_k \eta_k), \label{3.4.24}
\end{equation}
where $\mathcal{R}$ is the retraction operator and $t_k > 0$ is the step length.\\
The search direction $\eta_k$ used in the update step \eqref{3.4.24} is computed by
\begin{equation}
    \eta_k = -\mathrm{grad}J_2(V_k) + \beta_k^{DY} \mathcal{T}_{t_{k-1}\eta_{k-1}}(\eta_{k-1}), \quad k=1, 2, \dots \label{3.4.25}
\end{equation}
where $\mathcal{T}$ is the vector transport operator, and the Riemannian gradient $\mathrm{grad}J_2(V_k)$ is computed using \cref{lemma2}.\\
The Dai-Yuan parameter $\beta_k^{DY}$ is given by the expression
\begin{equation}
    \beta_k^{DY} = \frac{\|\mathrm{grad}J_2(V_k)\|^2_{V_k}}{\langle \mathrm{grad}J_2(V_k), \mathcal{T}_{t_{k-1}\eta_{k-1}}(\eta_{k-1}) \rangle_{V_k} - \langle \mathrm{grad}J_2(V_{k-1}), \eta_{k-1} \rangle_{V_{k-1}}}, \quad k=1,2,\dots \label{3.4.26}
\end{equation}
Here, $\langle\cdot,\cdot\rangle_{V_k}$ denotes the Riemannian metric on the tangent space $T_{V_k}St(n,r)$, and $\|\cdot\|_{V_k}$ is the norm induced by this metric.\\
Furthermore, following the approach in \cite{Sato2016}, we apply the following scaling to the vector transport operator $\mathcal{T}$ used in \eqref{3.4.25} and \eqref{3.4.26}
\begin{equation}
    \tilde{\mathcal{T}}_{t_{k-1}\eta_{k-1}}(\eta_{k-1}) = \min\left\{1, \frac{\|\eta_{k-1}\|_{V_{k-1}}}{\|\mathcal{T}_{t_{k-1}\eta_{k-1}}(\eta_{k-1})\|_{V_k}}\right\} \mathcal{T}_{t_{k-1}\eta_{k-1}}(\eta_{k-1}), \quad k=1, 2, \dots \label{3.4.27}
\end{equation}
This scaling in \eqref{3.4.27} ensures that
\[
\|\tilde{\mathcal{T}}_{t_{k-1}\eta_{k-1}}(\eta_{k-1})\|_{V_k} \le \|\eta_{k-1}\|_{V_{k-1}}.
\]
According to \cite{Sato2016}, this condition improves the convergence properties of the Dai-Yuan type Riemannian conjugate gradient method towards a critical point.
Once the search direction $\eta_k$ is determined, the step length $t_k$ for the Dai-Yuan Riemannian conjugate gradient method is determined by the Wolfe conditions.
Specifically, a step length $t_k > 0$ is chosen to satisfy the following two inequalities
\begin{align}
    J_2(\mathcal{R}_{V_k}(t_k \eta_k)) &\le J_2(V_k) + c_1 t_k \langle \mathrm{grad}J_2(V_k), \eta_k \rangle_{V_k}, \label{3.4.28} \\
    \langle \mathrm{grad}J_2(\mathcal{R}_{V_k}(t_k \eta_k)), \mathcal{T}_{t_k \eta_k}(\eta_k) \rangle_{\mathcal{R}_{V_k}(t_k \eta_k)} &\ge c_2 \langle \mathrm{grad}J_2(V_k), \eta_k \rangle_{V_k}, \label{3.4.29}
\end{align}
with the constants satisfying $0 < c_1 < c_2 < 1$.
To select a suitable step length $t_k$ that also ensures the asymptotic stability of the resulting reduced-order system, we employ a backtracking line search strategy.
Specifically, the step length is set to $t_k = \omega^{m_k}\gamma$, where $\omega \in (0,1)$ is a contraction factor, $\gamma > 0$ is an initial trial step length, and $m_k$ is the smallest non-negative integer for which the Wolfe conditions \eqref{3.4.28}, \eqref{3.4.29}, and the following stability condition are all satisfied
\begin{equation}
    \max_i \Re(\mathrm{eig}(\mathcal{R}_{V_k}(t_k \eta_k)^\top A \mathcal{R}_{V_k}(t_k \eta_k))) < 0, \label{3.4.30}
\end{equation}
where $\Re(\cdot)$ denotes the real part of a complex number and $\mathrm{eig}(\cdot)$ denotes the eigenvalues of a matrix.
The specific steps are outlined in Algorithm 4.2.
\par

\begin{algorithm}[H]
\caption{Stability-Preserving Riemannian Conjugate Gradient Method (SRCG-Stab) for 
$H_2$-Optimal MOR of LQO Systems.}
\label{alg2}
\begin{algorithmic}[1]
    \STATE \textbf{Input:} LQO system $\Sigma$.
    \STATE \textbf{Output:} Reduced-order LQO system $\hat{\Sigma}$ with matrices $(\hat{A}, \hat{B}, \hat{C}, \hat{M})$.
    \STATE \textbf{Initialize:} Choose an initial matrix $V_0 \in St(n,r)$.
    \STATE Compute a symmetric positive definite matrix $H$ satisfying the Lyapunov inequality $A^\top H + H A < 0$.
    \STATE Compute the Riemannian gradient $\mathrm{grad}J_2(V_0)$ using \eqref{3.4.6}.
    \STATE Set the initial search direction $\eta_0 \leftarrow - \mathrm{grad}J_2(V_0)$.
    \FOR{$k = 0, 1, \dots$ (until convergence)}
        \STATE Determine a step length $t_k$ satisfying the line search conditions \eqref{3.4.28}, \eqref{3.4.29}, and \eqref{3.4.30}.
        \STATE Update the iterate: $V_{k+1} \leftarrow \mathcal{R}_{V_k}(t_k \eta_k)$, where $\mathcal{R}_{V_k}$ is a retraction at $V_k$.
        \STATE Compute the gradient at the new iterate: $\mathrm{grad}J_2(V_{k+1})$.
        \STATE Transport the previous search direction to the new tangent space: $\mathcal{T}_{t_k \eta_k}(\eta_k)$, using \eqref{3.4.27}.
        \STATE Compute the Dai-Yuan parameter $\beta_{k+1}^{DY}$ using \eqref{3.4.26}.
        \STATE Update the search direction: $\eta_{k+1} \leftarrow - \mathrm{grad}J_2(V_{k+1}) + \beta_{k+1}^{DY}\mathcal{T}_{t_k \eta_k}(\eta_k)$.
    \ENDFOR
    \STATE \textbf{Return:} $\hat{A}=(V^\top HV)^{-1}V^\top HAV$, $\hat{B}=(V^\top 
    HV)^{-1}V^\top HB$, $\hat{C}=CV$, $\hat{M}=V^\top MV$.
\end{algorithmic}
\end{algorithm}
\par

In practice, one can choose the symmetric positive 
definite matrix $H$ in \cref{alg2} as the unique solution to the Lyapunov 
equation $A^\top H + H A + I_n = O$.

\section{Efficient Computation of the Riemannian Gradient}
\label{section4}

A primary computational bottleneck in the proposed Riemannian optimization algorithms is the need to solve the large-scale Sylvester equations for matrices $X$ and $K$ at each iteration.
Directly solving these $n \times r$ equations, where $n \gg r$, typically incurs a complexity of $O(n^3)$, which can be prohibitive.
To address this challenge, we construct an efficient approximation method based on the integral representations of the Sylvester equation solutions.
\par
The solution to the Sylvester equation for $X$, for instance, has the following integral representation
\[
    X = \int_0^\infty e^{At}B\hat{B}^\top e^{\hat{A}^\top t}\mathrm{d}t.
\]
Our approach is to approximate the matrix exponentials using truncated series of scaled Laguerre functions.
The matrix exponential can be expanded as $e^{At} = \sum_{k=0}^\infty 
A_k\phi_k^\alpha(t)$, where the coefficient matrices $\{A_k\}$ can be computed 
efficiently via the recurrence relation \cite{Xiao2022}
\begin{equation*}
    A_0 = \sqrt{2\alpha}(\alpha I+A)^{-1}, \quad A_i = (A-\alpha I)(A+\alpha I)^{-1} A_{i-1}, \quad i=1,2,\dots
\end{equation*}
By substituting truncated Laguerre expansions for $e^{At}$ and $e^{\hat{A}t}$ into the integral for $X$ and leveraging the orthonormality of the Laguerre functions, the integral can be evaluated analytically.
This procedure yields the following low-rank approximation for $X$
\begin{equation*}
    X \approx F \hat{F}^\top, \label{eq:X_approx}
\end{equation*}
where $F$ and $\hat{F}$ are block matrices containing the first $N$ Laguerre coefficients
\[
    F = \begin{pmatrix} A_0 B & A_1 B & \cdots & A_{N-1}B \end{pmatrix}, \quad
    \hat{F} = \begin{pmatrix} \hat{A}_0 \hat{B} & \hat{A}_1 \hat{B} & \cdots & \hat{A}_{N-1}\hat{B} \end{pmatrix}.
\]
An analogous procedure is applied to the integral form of $K$, yielding a similar low-rank approximation $K \approx -G \hat{G}^\top$, where
\begin{align}
	G&=\left(
	\begin{matrix}
		A_0^\top\begin{pmatrix}C^\top&\sqrt{2}MF\end{pmatrix}
		&A_1^\top\begin{pmatrix}C^\top&\sqrt{2}MF\end{pmatrix}
		&\cdots&A_{N-1}^\top\begin{pmatrix}C^\top&\sqrt{2}MF\end{pmatrix}
	\end{matrix}
	\right),\nonumber\\
	\hat{G}&=\left(
	\begin{matrix}
		\hat{A}_0^\top\begin{pmatrix}\hat{C}^\top&\sqrt{2}\hat{M}\hat{F}\end{pmatrix}
		&\hat{A}_1^\top\begin{pmatrix}\hat{C}^\top&\sqrt{2}\hat{M}\hat{F}\end{pmatrix}
		&\cdots
		&\hat{A}_{N-1}^\top\begin{pmatrix}\hat{C}^\top&\sqrt{2}\hat{M}\hat{F}\end{pmatrix}
	\end{matrix}
	\right).\nonumber
\end{align}
\par

The efficiency of this approach stems from the separation of computations.
The large matrices $F$ and $G$, which depend only on the full-order system, can be pre-computed once offline.
Within the optimization loop, one only needs to re-compute the small matrices $\hat{F}$ and $\hat{G}$, which depend on the current (small-dimensional) iterates.
This replaces the expensive $O(n^3)$ direct solve with efficient updates and matrix 
products of much smaller dimension, drastically reducing the computational workload 
per iterate.

\section{Numerical Experiments}
\label{section5}
This section evaluates the performance of the proposed methods, \cref{alg1} (GAAI) 
and \cref{alg2} (SRCG-Stab), through a numerical example.
All numerical simulations were performed in MATLAB\textsuperscript{\textregistered} 
R2024a (64-bit) on a computer equipped with an 
Intel\textsuperscript{\textregistered} Core\textsuperscript{TM} i9-13980HX 2.2-GHz 
processor, 32 GB of RAM.
To ensure reliable convergence, two stopping criteria are employed for the iterative algorithms.
First, the iteration terminates if the relative norm of the Riemannian gradient 
falls below a prescribed tolerance, where the relative gradient norm is defined as 
the ratio of the norm of the Riemannian gradient at current iterate and the
initial.
Second, the algorithm terminates upon reaching a predefined maximum number of iterations.
\par

This example is adapted from the methodology presented in \cite{Sato2016}, with several modifications to the implementation.
We set the system dimension to $n=300$ and the number of inputs to $m=1$.
The system matrix $A$ is constructed as the sum $A = A_{\mathrm{sym}} + A_{\mathrm{skew}}$, where $A_{\mathrm{sym}}$ is a randomly generated $n \times n$ symmetric negative definite matrix and $A_{\mathrm{skew}}$ is a randomly generated $n \times n$ skew-symmetric matrix.
The input matrix is set to a vector of all ones, $B = (1, 1, \ldots, 1)^\top$.
Similarly, the linear output vector is chosen as $C = (1, 1, \ldots, 1)$, corresponding to the sum of all state components.
For the quadratic part of the output, the identity matrix is used, $M=I$, which is symmetric and positive definite. Despite its simplicity, this choice of $M$ presents a challenging MOR scenario, as the identity matrix is not well-approximable by a low-rank matrix.
With the system matrices defined, a time-varying input signal $u(t) = \exp(\sin(2t))$ is used to excite the system dynamics.
The transient response of the LQO system to this input is then examined over the time interval $[0, 10]$.
\par

The system is reduced to order $r=10$ using the two methods proposed in this paper and two baseline algorithms for comparison.
The baseline methods are Krylov subspace method (denoted as KS, \cite{Antoulas2020}) and structure-preserving balanced truncation method (denoted as BT) from \cite{Benner2021}.
Prior to commencing the iterative process, the initial values are generated by applying the KS method to the original system.
For the GAAI algorithm (\cref{alg1}), the parameters are set as $\alpha_1=0.0001, \omega_1=0.3, \gamma_1=2$ and $\alpha_2=0.0001, \omega_2=0.3, \gamma_2=2$, with the process terminating after a fixed 16 iterations.
Figure \ref{fig1}(a) indicates that the algorithm has converged after 16 iterations, as the objective function value has plateaued and the norm of the Riemannian gradient has been reduced to a negligible level.
For the SRCG-Stab algorithm (\cref{alg2}), the line search parameters are set to $\omega = 0.7, \gamma = 0.01, c_1 = 0.01$, and $c_2 = 0.7$.
The algorithm terminates when the relative norm of the Riemannian gradient falls below $10^{-4}$.
Figure \ref{fig1}(b) illustrates the evolution of the objective function value and the relative gradient norm during the iterative process.
The objective function value decreases monotonically as the number of iterations increases, while the relative gradient norm shows a clear downward trend despite some fluctuations.
\par

\begin{figure}[ht]
	\centering
	\begin{minipage}{0.49\textwidth}
		\centering
		\includegraphics[width=\textwidth]{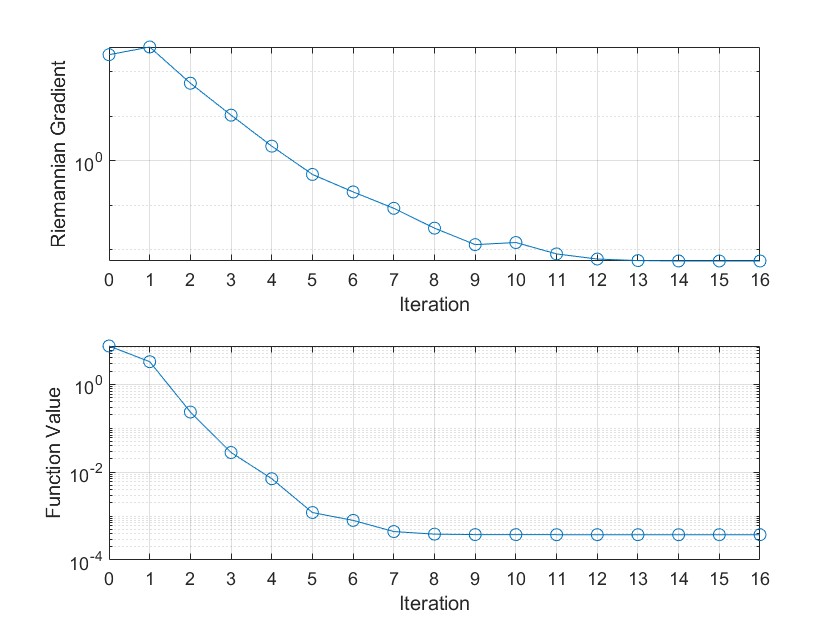}
		\caption*{(a)}
	\end{minipage}
	\hfill
	\begin{minipage}{0.49\textwidth}
		\centering
		\includegraphics[width=\textwidth]{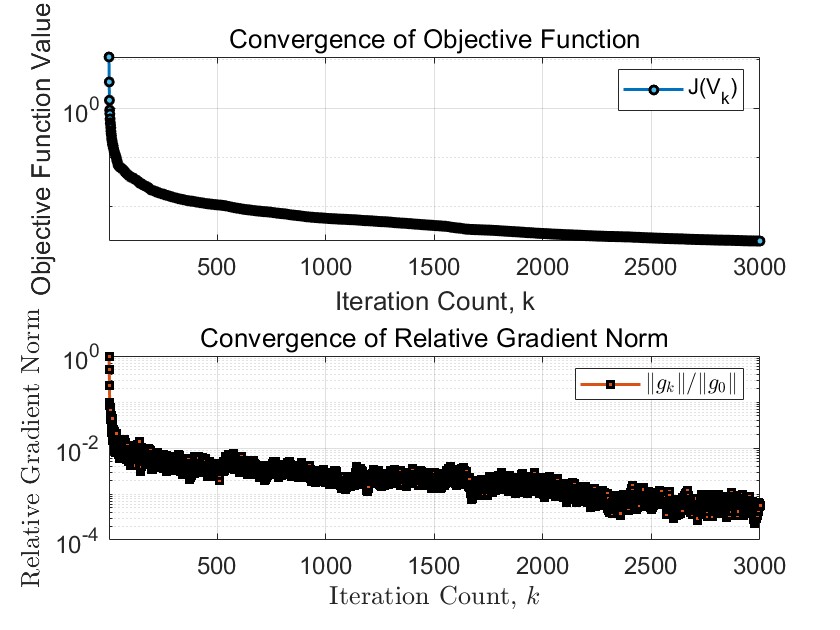}
		\caption*{(b)}
	\end{minipage}
	\caption{(a) The convergence history of the GAAI algorithm. (b) The evolution of the cost function value and the relative gradient norm for the SRCG-Stab algorithm.}
	\label{fig1}
\end{figure}

A time-domain comparison of the reduced-order systems generated by each algorithm against the original LQO system is presented, with the results depicted in Figure \ref{fig2}.
As depicted in the response plots, for this example with a reduced order of $r=10$, the models generated by the BT, GAAI, and SRCG-Stab methods are almost visually indistinguishable from the original system, in contrast to the Krylov subspace method.
This indicates that these three methods effectively capture the dominant dynamic characteristics of the full-order system.
To provide a clearer distinction, Figure \ref{fig2}(b) plots the relative error of each reduced-order model's output.
The GAAI algorithm consistently achieves the lowest relative error across the entire time interval, with a magnitude ranging between $10^{-5}$ and $10^{-4}$.
This result highlights its superior performance, achieving approximately an order of magnitude higher accuracy than the already high-performance BT method.
The SRCG-Stab algorithm yields a relative error of approximately $10^{-2}$. 
While less accurate than GAAI and BT, the SRCG-Stab method is significantly better 
than the Krylov method for this example, along with a  
guarantee of stability in the iteration. 
\par

\begin{figure}[ht]
	\centering
	\begin{minipage}{0.49\textwidth}
		\centering
		\includegraphics[width=\textwidth]{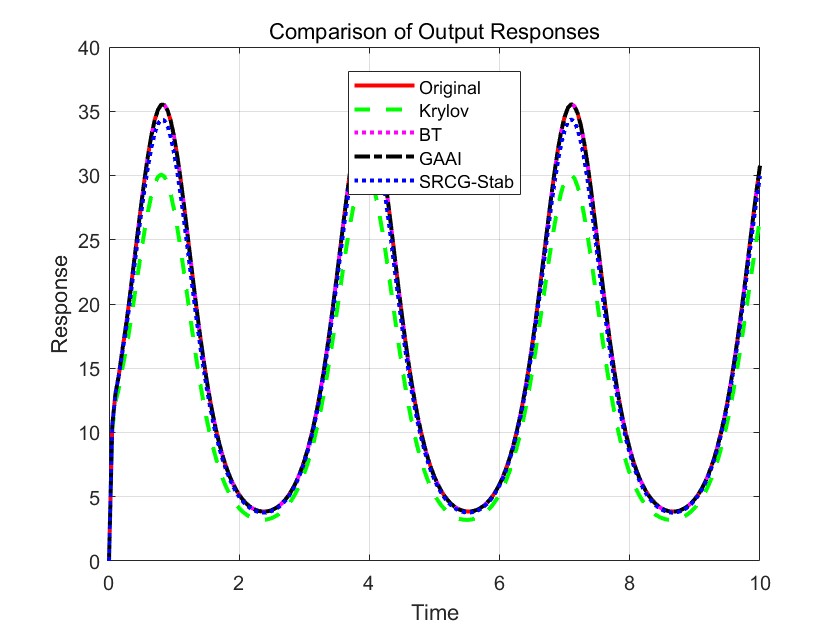}
		\caption*{(a)}
	\end{minipage}
	\hfill
	\begin{minipage}{0.49\textwidth}
		\centering
		\includegraphics[width=\textwidth]{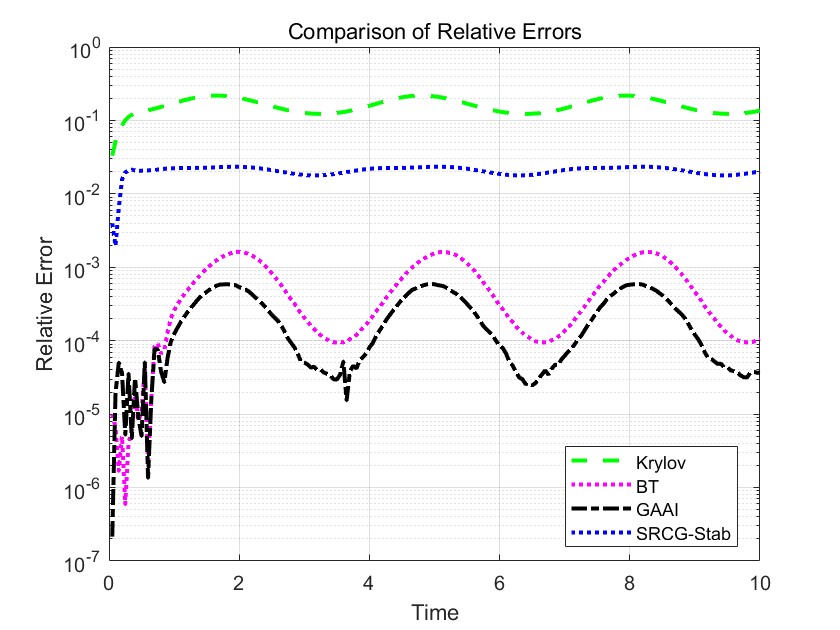}
		\caption*{(b)}
	\end{minipage}
	\caption{(a) Comparison of output responses (b) Comparison of relative errors.}
	\label{fig2}
\end{figure}

\section{Conclusions}
\label{section6}

This paper investigates the $H_2$-optimal MOR of LQO systems within the Petrov-Galerkin projection framework.
Two primary Riemannian optimization-based algorithms are developed.
One formulates the $H_2$ optimal MOR problem as an unconstrained minimization on the 
Grassmann manifold, for which a bivariable alternating optimization algorithm is 
proposed.
The other imposes a specific constraint on the projection matrices, and the 
stability can be guaranteed naturally in the iteration on the Stiefel manifold.
A approximate solver based on the Laguerre expansion improve the efficiency of the 
proposed algorithms dramatically. 
The performance of the two methods is validated through numerical experiments.

\addcontentsline{toc}{section}{Reference}
\markboth{Reference}{}
\bibliographystyle{elsarticle-num-names}
\bibliography{reference}

\end{document}